\documentclass[10pt,twoside]{amsart} 
\usepackage{amsmath, amsthm, amscd, amsfonts, amssymb, graphicx, color}
\usepackage[bookmarksnumbered, colorlinks, plainpages]{hyperref}

\newtheorem{theorem}{Theorem}[section]
\newtheorem{lemma}[theorem]{Lemma}
\newtheorem{proposition}[theorem]{Proposition}
\newtheorem{corollary}[theorem]{Corollary}

\theoremstyle{definition}
\newtheorem{definition}[theorem]{Definition}
\newtheorem{example}[theorem]{Example}

\newtheorem{remark}[theorem]{Remark}

\numberwithin{equation}{section}

\def\<{\langle}
\def\>{\rangle}

\long\def\alert#1{\smallskip{\hskip\parindent\vrule%
\vbox{\advance\hsize-2\parindent\hrule\smallskip\parindent.4\parindent%
\narrower\noindent#1\smallskip\hrule}\vrule\hfill}\smallskip}

\begin{document}

\title[\ci\ $2$-ABSORBING PRIMARY SUBMODULES]{ON $2$-ABSORBING PRIMARY SUBMODULES OF MODULES OVER COMMUTATIVE RINGS}

\author{Hojjat Mostafanasab, Ece Yetkin, \"{U}nsal Tekir \\ 
and Ahmad Yousefian Darani$^{*}$}
\address{Department of Mathematics and Applications, University of Mohaghegh Ardabili, P. O. Box 179, Ardabil, Iran} \email{h.mostafanasab@uma.ac.ir and yousefian@uma.ac.ir}

\address{Department of Mathematics, Marmara University, Istanbul, Turkey}
\email{utekir@marmara.edu.tr and ece.yetkin@marmara.edu.tr}
\thanks{$^*$Corresponding author}
\thanks{Support information for the other authors.}

\thanks{{\scriptsize
\hskip -0.4 true cm 2010 {\it Mathematics Subject Classification}.
Primary 13A15; Secondary 13F05, 13G05.
}}

\begin{abstract}
All rings are commutative with $1\neq0$, and all modules are unital. The
purpose of this paper is to investigate the concept of $2$-absorbing primary
submodules generalizing $2$-absorbing primary ideals of rings. Let $M$ be an 
$R$-module. A proper submodule $N$ of an $R$-module $M$ is called a \textit{%
2-absorbing primary submodule} of $M$ if whenever $a,b\in R$ and $m\in M$
and $abm\in N$, then $am\in M$-$\mathrm{rad}(N)$ or $bm\in M$-$\mathrm{rad}%
(N)$ or $ab\in(N:_{R}M)$. It is shown that a proper submodule $N$ of $M$ is
a 2-absorbing primary submodule if and only if whenever $I_{1}I_{2}K%
\subseteq N$ for some ideals $I_{1},I_{2}$ of $R$ and some submodule $K$ of $%
M$, then $I_{1}I_{2}\subseteq(N:_{R}M)$ or $I_{1}K\subseteq M$-${rad}(N)$ or 
$I_{2}K\subseteq M$-${rad}(N)$. We prove that for a submodule $N$ of an $R$%
-module $M$ if $M$-$rad(N)$ is a prime submodule of $M$, then $N$ is a
2-absorbing primary submodule of $M$. If $N$ is a 2-absorbing primary
submodule of a finitely generated multiplication $R$-module $M$, then $%
(N:_{R}M)$ is a 2-absorbing primary ideal of $R$ and $M$-${rad}(N)$ is a
2-absorbing submodule of $M$.
\end{abstract}

\maketitle

\baselineskip 17 pt

\section{Introduction and Preliminaries}

Throughout this paper all rings are commutative with a nonzero identity and
all modules are considered to be unitary. Prime submodules have an important
role in the theory of modules over commutative rings. Let $M$ be a module
over a commutative ring $R$. A \textit{prime} (resp. \textit{primary})
submodule is a proper submodule $N$ of $M$ with the property that for $a\in R
$ and $m\in M$, $am\in N$ implies that $m\in N$ or $a\in (N:_{R}M)$ (resp. $%
a^{k}\in (N:_{R}M)$ for some positive integer $k$). In this case $p=(N:_{R}M)
$ (resp. $p=\sqrt{(N:_{R}M)}$) is a prime ideal of $R$. There are several
ways to generalize the concept of prime submodules. Weakly prime submodules
were introduced by Ebrahimi Atani and Farzalipour in \cite{E1}. A proper
submodule $N$ of $M$ is \textit{weakly prime} if for $a\in R$ and $m\in M$
with $0\neq am\in N$, either $m\in N$ or $a\in (N:_{R}M)$. Behboodi and
Koohi in \cite{BK} defined another class of submodules and called it weakly
prime. Their paper is on the basis of some recent papers devoted to this new
class of submodules. Let $R$ be a ring and $M$ an $R$-module. A proper
submodule $N$ of $M$ is said to be \textit{weakly prime} when for $a,b\in R$
and $m\in M$, $abm\in N$ implies that $am\in N$ or $bm\in N$. To avoid the
ambiguity, Behboodi renamed this concept and called submodules introduced in 
\cite{BK}, \emph{classical prime submodule}.

Badawi in \cite{B} generalized the concept of prime ideals in a different
way. He defined a nonzero proper ideal $I$ of $R$ to be a \textit{%
2-absorbing ideal} of $R$ if whenever $a, b, c\in R$ and $abc\in I$, then $%
ab\in I$ or $ac\in I$ or $bc\in I$. This definition can obviously be made
for any ideal of $R$. This concept has a generalization, called weakly $2$%
-absorbing ideals, which has studied in \cite{YB}. A proper ideal $I$ of $R$
to be a \textit{weakly 2-absorbing ideal of} $R$ if whenever $a, b, c\in R$
and $0\neq abc\in I$, then $ab\in I$ or $ac\in I$ or $bc\in I$. Anderson and
Badawi \cite{AB1} generalized the concept of $2$-absorbing ideals to $n$%
-absorbing ideals. According to their definition, a proper ideal $I$ of $R$
is called an $n$-\textit{absorbing} (resp. \textit{strongly $n$-absorbing})
ideal if whenever $x_1\cdots x_{n+1}\in I$ for $x_1,...,x_{n+1}\in R$ (resp. 
$I_1\cdots I_{n+1}\subseteq I$ for ideals $I_1, \cdots I_{n+1}$ of $R$),
then there are $n$ of the $x_i$'s (resp. $n$ of the $I_i$'s) whose product
is in $I$. They proved that a proper ideal $I$ of $R$ is 2-absorbing if and
only if $I$ is strongly 2-absorbing.

In \cite{YF1}, the concept of $2$-absorbing and weakly $2$-absorbing ideals
generalized to submodules of a module over a commutative ring. Let $M$ be an 
$R$-module and $N$ a proper submodule of $M$. $N$ is said to be a \textit{%
2-absorbing submodule} (resp. \textit{weakly 2-absorbing submodule}) of $M$
if whenever $a,b\in R$ and $m\in M$ with $abm\in N$ (resp. $0\neq abm\in N$%
), then $ab\in (N:_{R}M)$ or $am\in N$ or $bm\in N$. Badawi et. al. in \cite%
{Bt} introduced the concept of $2$-absorbing primary ideals, where a proper
ideal $I$ of $R$ is called \textit{2-absorbing primary} if whenever $%
a,b,c\in R$ with $abc\in I$, then $ab\in I$ or $ac\in \sqrt{I}$ or $bc\in 
\sqrt{I}$.

Let $R$ be a ring, $M$ an $R$-module and $N$ a submodule of $M$. We will
denote by $(N:_{R}M)$ \textit{the residual of $N$ by} $M$, that is, the set
of all $r\in R$ such that $rM\subseteq N$. The annihilator of $M$ which is
denoted by $ann_{R}(M)$ is $(0:_{R}M)$. An $R$-module $M$ is called a 
\textit{multiplication module} if every submodule $N$ of $M$ has the form $IM
$ for some ideal $I$ of $R$. Note that, since $I\subseteq (N:_{R}M)$ then $%
N=IM\subseteq (N:_{R}M)M\subseteq N$. So that $N=(N:_{R}M)M$ \cite{ES}.
Finitely generated faithful multiplication modules are cancellation modules 
\cite[ Corollary to Theorem 9]{S}, where an $R$-module $M$ is defined to be
a \textit{cancellation module} if $IM=JM$ for ideals $I$ and $J$ of $R$
implies $I=J$. It is well-known that if $R$ is a commutative ring and $M$ a
nonzero multiplication $R$-module, then every proper submodule of $M$ is
contained in a maximal submodule of M and $K$ is a maximal submodule of $M$
if and only if there exists a maximal ideal $\mathfrak{m}$ of $R$ such that $%
K=\mathfrak{m}M$ \cite[Theorem 2.5]{ES}. If $M$ is a finitely generated
faithful multiplication $R$-module (hence cancellation), then it is easy to
verify that $(IN:_{R}M)=I(N:_{R}M)$ for each submodule $N$ of $M$ and each
ideal $I$ of $R$. For a submodule $N$ of $M$, if $N=IM$ for some ideal $I$
of $R$, then we say that $I$ is a presentation ideal of $N$. Clearly, every
submodule of $M$ has a presentation ideal if and only if $M$ is a
multiplication module. Let $N$ and $K$ be submodules of a multiplication $R$%
-module $M$ with $N=I_{1}M$ and $K=I_{2}M$ for some ideals $I_{1}$ and $I_{2}
$ of $R$. The product of $N$ and $K$ denoted by $NK$ is defined by $%
NK=I_{1}I_{2}M$. Then by \cite[Theorem 3.4]{Am}, the product of $N$ and $K$
is independent of presentations of $N$ and $K$. Moreover, for $a,b\in M$, by 
$ab$, we mean the product of $Ra$ and $Rb$. Clearly, $NK$ is a submodule of $%
M$ and $NK\subseteq N\cap K$ (see \cite{Am}). Let $N$ be a proper submodule
of a nonzero $R$-module $M$. Then the $M$-radical of $N$, denoted by $M$-$%
\mathrm{rad}(N)$, is defined to be the intersection of all prime submodules
of $M$ containing $N$. If $M$ has no prime submodule containing $N$, then we
say $M$-$\mathrm{rad}(N)=M$. It is shown in \textrm{\cite[Theorem 2.12]{ES}}
that if $N$ is a proper submodule of a multiplication $R$-module $M$, then $M
$-$\mathrm{rad}(N)=\sqrt{(N:_{R}M)}M$. In this paper we define the concept
of $2$-absorbing primary submodules. We give some basic results of this
class of submodules and discuss on the relations among $2$-absorbing ideals, 
$2$-absorbing submodules, $2$-absorbing primary ideals and $2$-absorbing
primary submodules.

\section{Properties of $2$-absorbing primary submodules}

\begin{definition}
A proper submodule $N$ of an $R$-module $M$ is called a \textit{2-absorbing
primary submodule} (resp. \textit{weakly 2-absorbing primary submodule}) of $%
M$ if whenever $a,b\in R$ and $m\in M$ and $abm\in N$ (resp. $0\neq abm\in N)
$, then $am\in M$-$\mathrm{rad}(N)$ or $bm\in M$-$\mathrm{rad}(N)$ or $%
ab\in(N:_{R}M)$.
\end{definition}

\begin{example}
Let $p$ be a fixed prime integer and $N_{0}=\mathbb{%
\mathbb{N}
}\cup \{0\}$. Each proper $\mathbb{Z}$-submodule of $\mathbb{Z}(p^{\infty })$
is of the form $G_{t}=\langle 1/p^{t}+\mathbb{Z}\rangle $ for some $t\in
N_{0}$. In \textrm{\cite[Example 1]{EC} it was shown that every submodule $%
G_{t}$ is not primary. For each $t\in N_{0}$, $(G_{t}:_{\mathbb{Z}}\mathbb{Z}%
(p^{\infty }))=0$. Note that $p^{2}\left( \frac{1}{p^{t+2}}+\mathbb{Z}%
\right) =\frac{1}{p^{t}}+\mathbb{Z}\in G_{t}$, but neither $p^{2}\in
(G_{t}:_{\mathbb{Z}}\mathbb{Z}(p^{\infty }))=0$ nor $p\left( \frac{1}{p^{t+2}%
}+\mathbb{Z}\right) \in G_{t}$. Hence $\mathbb{Z}(p^{\infty })$ has no $2$%
-absorbing submodule. Since every prime submodule is $2$-absorbing, then $%
\mathbb{Z}(p^{\infty })$ has no prime submodule. Therefore $\mathbb{Z}%
(p^{\infty })$-$\mathrm{rad}(G_{t})=\mathbb{Z}(p^{\infty })$, and so $G_{t}$
is a 2-absorbing primary submodule of $\mathbb{Z}(p^{\infty })$. }
\end{example}

\begin{theorem}
Let $N$ be a proper submodule of an $R$-module $M$. Then the following
conditions are equivalent:

\begin{enumerate}
\item $N$ is a 2-absorbing primary submodule of $M$;

\item For every elements $a,b\in R$ such that $ab\notin(N:_{R}M)$, $%
(N:_{M}ab)\subseteq(M$-$rad(N):_{M}a)\cup(M$-$rad(N):_{M}b);$

\item For every elements $a,b\in R$ such that $ab\notin(N:_{R}M)$, $%
(N:_{M}ab)\subseteq(M$-$rad(N):_{M}a)$ or $(N:_{M}ab)\subseteq(M$-$%
rad(N):_{M}b).$
\end{enumerate}
\end{theorem}

\begin{proof}
(1)$\Rightarrow$(2) Suppose that $a,b\in R$ such that $ab\notin(N:_{R}M)$.
Let $m\in(N:_{M}ab)$. Then $abm\in N$, and so either $ma\in M$-$rad(N)$ or $%
bm\in M$-$rad(N)$. Therefore either $m\in(M$-$rad(N):_{M}a)$ or $m\in(M$-$%
rad(N):_{M}b)$. Hence $(N:_{M}ab)\subseteq(M$-$rad(N):_{M}a)\cup(M$-$%
rad(N):_{M}b)$.\newline
(2)$\Rightarrow$(3) Notice to the fact that if a submodule (a subgroup) is a
subset of the union of two submodules (two subgroups), then it is a subset
of one of them. Thus we have $(N:_{M}ab)\subseteq(M$-$rad(N):_{M}a)$ or $%
(N:_{M}ab)\subseteq(M$-$rad(N):_{M}b).$\newline
(3)$\Rightarrow$(1) is straightforward.
\end{proof}

\begin{lemma}
\label{finmul} Let $M$ be a finitely generated multiplication $R$-module.
Then for any submodule $N$ of $M$, $\sqrt{(N:_{R}M)}=(M$-${rad}(N):_{R}M)$.
\end{lemma}

\begin{proof}
By \textrm{\cite[Theorem 4]{MM}}, $(M$-$\mathrm{rad}(N):_{R}M)\subseteq\sqrt{%
(N:_{R}M)}$. Now we prove the other containment without any assumption on $M$%
. Let $K$ be a prime submodule of $M$ containing $N$. Then clearly $(K:M)$
is a prime ideal that contains $(N:M)$. Therefore $\sqrt{(N:_{R}M)}%
\subseteq(K:M)$, so $\sqrt{(N:_{R}M)}\subseteq(M$-$\mathrm{rad}(N):_{R}M)$.
\end{proof}

\begin{proposition}
\label{finmul2} Let $M$ be a finitely generated multiplication $R$-module
and $N$ be a submodule of $M$. Then $M$-$rad(N)$ is a primary submodule of $M
$ if and only if$\ M$-$rad(N)$ is a prime submodule of $M$.
\end{proposition}

\begin{proof}
Suppose that $M$-$\mathrm{rad}(N)$ is a primary submodule of $M$. Let $a\in R
$ and $m\in M$ be such that $am\in M$-$\mathrm{rad}(N)$ and $m\not\in M$-$%
\mathrm{rad}(N).$ Since $M$-$\mathrm{rad}(N)$ is primary, it follows $a\in 
\sqrt{(M\text{-}\mathrm{rad}(N):_{R}M)}=\sqrt{\sqrt{(N:_{R}M)}}=\newline
\sqrt{(N:_{R}M)}=(M$-$\mathrm{rad}(N):_{R}M)$, by Lemma \ref{finmul}. Thus $M
$-$\mathrm{rad}(N)$ is a prime submodule of $M.$ The converse part is clear.
\end{proof}

\begin{theorem}
\label{m-rad} Let $M$ be a finitely generated multiplication $R$-module. If $%
N$ is a 2-absorbing primary submodule of $M$, then

\begin{enumerate}
\item $(N:_{R}M)$ is a 2-absorbing primary ideal of $R$.

\item $M$-${rad}(N)$ is a 2-absorbing submodule of $M$.
\end{enumerate}
\end{theorem}

\begin{proof}
$(1)$ Let $a,b,c\in R$ be such that $abc\in (N:_{R}M)$, $ac\not\in \sqrt{%
(N:_{R}M)}$ and $bc\not\in \sqrt{(N:_{R}M)}$. Since, by Lemma \ref{finmul}, $%
\sqrt{(N:_{R}M)}=(M$-$\mathrm{rad}(N):_{R}M)$, there exist $m_{1},m_{2}\in M$
such that $acm_{1}\not\in M$-$\mathrm{rad}(N)$ and $bcm_{2}\not\in M$-$%
\mathrm{rad}(N)$. But $ab(cm_{1}+cm_{2})\in N$, because $abc\in (N:_{R}M)$.
So $a(cm_{1}+cm_{2})\in M$-$\mathrm{rad}(N)$ or $b(cm_{1}+cm_{2})\in M$-$%
\mathrm{rad}(N)$ or $ab\in (N:_{R}M)$, since $N$ is $2$-absorbing primary.
If $ab\in (N:_{R}M)$, then we are done. Thus assume that $%
a(cm_{1}+cm_{2})\in M$-$\mathrm{rad}(N)$. So $acm_{2}\not\in M$-$\mathrm{rad}%
(N)$, because $acm_{1}\not\in M$-$\mathrm{rad}(N)$. Therefore $ab\in
(N:_{R}M)$, since $N$ is $2$-absorbing primary and $abcm_{2}\in N$.
Similarly if $b(cm_{1}+cm_{2})\in M$-$\mathrm{rad}(N)$, then $ab\in (N:_{R}M)
$. Consequently $(N:_{R}M)$ is a $2$-absorbing primary ideal.\newline
$(2)$ By \textrm{\cite[Theorem 2.3]{Bt}} we have two cases.

\textbf{Case 1}. $\sqrt{(N:_{R}M)}=p$ is a prime ideal of $R$. Since $M$ is
a multiplication module, $M$-$\mathrm{rad}(N)=\sqrt{(N:_{R}M)}M=pM$, where $%
pM$ is a prime submodule of $M$ by \textrm{\cite[Corollary 2.11]{ES}}. Hence
in this case $M$-$\mathrm{rad}(N)$ is a $2$-absorbing submodule of $M$.

\textbf{Case 2}. $\sqrt{(N:_{R}M)}=p_{1}\cap p_{2}$, where $p_{1},~p_{2}$
are distinct prime ideals of $R$ that are minimal over $(N:_{R}M)$. In this
case, we have $M$-$\mathrm{rad}(N)=\sqrt{(N:_{R}M)}M=(p_{1}\cap
p_{2})M=([p_{1}+\mathrm{ann}M]\cap[p_{2}+\mathrm{ann}M])M=p_{1}M \cap p_{2}M$%
, where $p_{1}M$, $p_{2}M$ are prime submodules of $M$ by \textrm{\cite[%
Corollary 2.11, 1.7]{ES}}. Consequently, $M$-$\mathrm{rad}(N)$ is a $2$%
-absorbing submodule of $M$ by \textrm{\cite[Theorem 2.3]{YF1}}.
\end{proof}

\begin{theorem}
\label{rad} Let $M$ be a (resp. finitely generated multiplication) $R$%
-module and $N$ be a submodule of $M$. If $M$-$rad(N)$ is a (resp. primary)
prime submodule of $M$, then $N$ is a 2-absorbing primary submodule of $M$.
\end{theorem}

\begin{proof}
Suppose that $M$-$\mathrm{rad}(N)$ is a prime submodule of $M$. Let $a,b\in R
$ and $m\in M$ be such that $abm\in N$, $am\not\in M$-$\mathrm{rad}(N)$.
Since $M$-$\mathrm{rad}(N)$ is a prime submodule and $abm\in M$-$\mathrm{rad}%
(N)$, then $b\in (M$-$\mathrm{rad}(N):_{R}M)$. So $bm\in M$-$\mathrm{rad}(N)$%
. Consequently $N$ is a $2$-absorbing primary submodule of $M$. Now assume
that $M$ is a finitely generated multiplication module and $M$-$\mathrm{rad}%
(N)$ is a primary submodule of $M$, then $M$-$\mathrm{rad}(N)$ is a prime
submodule of $M$, by Proposition \ref{finmul2}. Therefore $N$ is 2-absorbing
primary.
\end{proof}

In \textrm{\cite[Theorem 1(3)]{Al}}, it was shown that for any faithful
multiplication module $M$ not necessary finitely generated, $M$-$\mathrm{rad}%
(IM)=\sqrt{I}M$ for any ideal $I$ of $R$.

\begin{theorem}
Let $M$ be a (resp. finitely generated faithful multiplication) faithful
multiplication $R$-module. If $M$-${rad}(N)$ is a (resp. primary) prime
submodule of $M$, then $N^{n}$ is a 2-absorbing primary submodule of $M$ for
every positive integer $n\geq1$.
\end{theorem}

\begin{proof}
Assume that $M$ is a $($resp. finitely generated faithful multiplication$)$
faithful multiplication module and $M$-$\mathrm{rad}(N)$ is a $($resp.
primary$)$ prime submodule of $M$. There exists an ideal $I$ of $R$ such
that $N=IM$. Thus 
\begin{equation*}
M-\mathrm{rad}(N^{n})=\sqrt{I^{n}}M=M-\mathrm{rad}(N),
\end{equation*}

which is a $($resp. primary$)$ prime submodule of $M$. Hence for every
positive integer $n\geq1$, $N^{n}$ is a $2$-absorbing primary submodule of $M
$, by Theorem \ref{rad}.
\end{proof}

Recall that a commutative ring $R$ with $1\neq0$ is called a divided ring if
for every prime ideal $p$ of $R$, we have $p\subseteq xR$ for every $x\in
R\backslash p$. Generalizing this idea to modules we say that an $R$-module $%
M$ is divided if for every prime submodule $N$ of $M$, $N\subseteq Rm$ for
all $m\in M\backslash N$.

\begin{theorem}
If $M$ is a divided $R$-module, then every proper submodule of $M$ is a
2-absorbing primary submodule of $M$. In particular, every proper submodule
of a chained module is a 2-absorbing primary submodule.
\end{theorem}

\begin{proof}
Let $N$ be a proper submodule of $M$. Since the prime submodules of a
divided module are linearly ordered, we conclude that $M$-$\mathrm{rad}(N)$
is a prime submodule of $M$. Hence $N$ is a $2$-absorbing primary submodule
of $M$ by Theorem \ref{rad}.
\end{proof}

\begin{remark}
\label{fg1} Let $I=(0:_{R}M)$ and $R^{\prime}=R/I$. It is easy to see that $%
N $ is a $2$-absorbing primary $R$-submodule of $M$ if and only if $N$ is a $%
2$-absorbing primary $R^{\prime}$-submodule of $M$. Also, $(N:_{R}M)$ is a $%
2 $-absorbing primary ideal of $R$ if and only if $(N:_{R^{\prime}}M)$ is a $%
2$-absorbing primary ideal of $R^{\prime}$.
\end{remark}




\begin{theorem}
Let $S$ be a multiplicatively closed subset of $R$ and $M$ be an $R$-module.
If $N$ is a 2-absorbing primary submodule of $M$ and $S^{-1}N\not=S^{-1}M$,
then $S^{-1}N$ is a 2-absorbing primary submodule of $S^{-1}M$. 
\end{theorem}

\begin{proof}
If $\frac{a_{1}}{s_{1}}\frac{a_{2}}{s_{2}}\frac{m}{s}\in S^{-1}N$, then $%
ua_{1}a_{2}m\in N$ for some $u\in S$. It follows that $ua_{1}m\in M$-$%
\mathrm{rad}(N)$ or $ua_{2}m\in M$-$\mathrm{rad}(N)$ or $a_{1}a_{2}\in
(N:_{R}M)$, so we conclude that $\frac{a_{1}}{s_{1}}\frac{m}{s}=\frac{ua_{1}m%
}{us_{1}s}\in S^{-1}(M$-$\mathrm{rad}(N))\subseteq S^{-1}M$-$\mathrm{rad}%
(S^{-1}N)$ or $\frac{a_{2}}{s_{2}}\frac{m}{s}=\frac{ua_{2}m}{us_{2}s}\in
S^{-1}M$-$\mathrm{rad}(S^{-1}N)$ or $\frac{a_{1}}{s_{1}}\frac{a_{2}}{s_{2}}=%
\frac{a_{1}a_{2}}{s_{1}s_{2}}\in
S^{-1}(N:_{R}M)\subseteq(S^{-1}N:_{S^{-1}R}S^{-1}M)$.\newline
\end{proof}





\begin{theorem}
\label{absfait} Let $I$ be a 2-absorbing primary ideal of a ring $R$ and $M$
a faithful multiplication $R$-module such that ${Ass}_{R}(M/{\sqrt{I}M})$ is
a totally ordered set. Then $abm\in IM$ implies that $am\in\sqrt{I}M$ or $%
bm\in\sqrt{I}M$ or $ab\in I$ whenever $a,b\in R$ and $m\in M$.
\end{theorem}

\begin{proof}
Let $a,b\in R$, $m\in M$ and $abm\in IM$. If $(\sqrt{I}M:_{R}am)=R$ or $(%
\sqrt{I}M:_{R}bm)=R$, we are done. Suppose that $(\sqrt{I}M:_{R}am)$ and $(%
\sqrt{I}M:_{R}bm)$ are proper ideals of $R$. Since $\mathrm{Ass}_{R}(M/{%
\sqrt{I}M})$ is a totally ordered set, $(\sqrt{I}M:_{R}am)\cup (\sqrt{I}%
M:_{R}bm)$ is an ideal of $R$, and so there is a maximal ideal $\mathfrak{m}$
such that $(\sqrt{I}M:_{R}am)\cup (\sqrt{I}M:_{R}bm)\subseteq \mathfrak{m}$.
We have $am\notin \mathrm{T}_{\mathfrak{m}}(M):=\{m^{\prime }\in
M:(1-x)m^{\prime }=0$, for some $x\in \mathfrak{m}\}$, since $am\in \mathrm{T%
}_{\mathfrak{m}}(M)$ implies that $(1-x)am=0$ for some $x\in \mathfrak{m}$,
thus $(1-x)am\in \sqrt{I}M$ and so $1-x\in (\sqrt{I}M:_{R}am)\subseteq 
\mathfrak{m}$, a contradiction. So by \textrm{\cite[Theorem 1.2]{ES}}, there
are $x\in \mathfrak{m}$ and $m^{\prime }\in M$ such that $(1-x)M\subseteq
Rm^{\prime }$. Thus, $(1-x)m=rm^{\prime }$ some $r\in R$. Moreover, $%
(1-x)abm=sm^{\prime }$ for some $s\in I$, because $abm\in IM$. Hence $%
(abr-s)m^{\prime }=0$ and so $(1-x)(abr-s)M\subseteq (abr-s)Rm^{\prime }=0$.
Thus $(1-x)(abr-s)=0$, because $M$ is faithful. Therefore, $%
(1-x)abr=(1-x)s\in I$. Then $(1-x)ar\in \sqrt{I}$ or $(1-x)b\in \sqrt{I}$ or 
$abr\in I$, since $I$ is 2-absorbing primary.\newline
If $(1-x)ar\in \sqrt{I}$, then $(1-x)a\in \sqrt{I}$ or $(1-x)r\in \sqrt{I}$
or $ar\in \sqrt{I}$, because by \textrm{\cite[Theorem 2.2]{Bt}} $\sqrt{I}$
is a $2$-absorbing ideal of $R$. If $(1-x)a\in \sqrt{I}$, then $(1-x)am\in 
\sqrt{I}M$ and so $1-x\in (\sqrt{I}M:_{R}am)\subseteq \mathfrak{m}$ that is
a contradiction. If $(1-x)r\in \sqrt{I}$, then $(1-x)^{2}m=(1-x)rm^{\prime
}\in \sqrt{I}M$ which implies that $(1-x)^{2}\in (\sqrt{I}M:_{R}m)\subseteq (%
\sqrt{I}M:_{R}am)\subseteq \mathfrak{m}$, a contradiction. Similarly we can
see that $(1-x)b\not\in \sqrt{I}$. Now, $ar\in \sqrt{I}$ implies that $%
(1-x)am=arm^{\prime }\in \sqrt{I}M$ and so $1-x\in (\sqrt{I}%
M:_{R}am)\subseteq \mathfrak{m}$ which is a contradiction.\newline
If $arb\in I$, then $ar\in \sqrt{I}$ or $br\in \sqrt{I}$ or $ab\in I$ which
the first two cases are impossible, thus $ab\in I$.
\end{proof}

Let $R$ be a ring with the total quotient ring $K$. A nonzero ideal $I$ of $R
$ is said to be \textit{invertible} if $II^{-1}=R$, where $I^{-1}=\{x\in
K\mid xI\subseteq R\}$. The concept of an invertible submodule was
introduced in \cite{NA} as a generalization of the concept of an invertible
ideal. Let $M$ be an $R$-module and let $S=R \backslash\{0\}$. Then $T=\{t
\in S\mid tm=0 \mbox{ for some } m\in M \mbox{ implies } m=0\}$ is a
multiplicatively closed subset of $R$. Let $N$ be a submodule of $M$ and $%
N^{\prime}=\{x\in R_{T}\mid xN\subseteq M\}$. A submodule $N$ is said to be 
\textit{invertible in} $M$, if $N^{\prime}N=M$, \cite{NA}. A nonzero $R$%
-module $M$ is called \textit{Dedekind} provided that each nonzero submodule
of $M$ is invertible.

We recall from \cite{Kh} that, a finitely generated torsion-free
multiplication module $M$ over a domain $R$ is a Dedekind module if and only
if $R$ is a Dedekind domain. 

\begin{theorem}
Let $R$ be a Noetherian domain, $M$ a torsion-free multiplication $R$%
-module. Then the following statements are equivalent:

\begin{enumerate}
\item $M$ is a Dedekind module;

\item If $N$ is a nonzero 2-absorbing primary submodule of $M$, then either $%
N=\mathfrak{M}^{n}$ for some maximal submodule $\mathfrak{M}$ of $M$ and
some positive integer $n\geq1$ or $N=\mathfrak{M}_{1}^{n}\mathfrak{M}_{2}^{m}
$ for some maximal submodules $\mathfrak{M}_{1}$ and $\mathfrak{M}_{2}$ of $M
$ and some positive integers $n,m\geq1$;

\item If $N$ is a nonzero 2-absorbing primary submodule of $M$, then either $%
N=P^{n}$ for some prime submodule $P$ of $M$ and some positive integer $%
n\geq1$ or $N=P_{1}^{n}P_{2}^{m}$ for some prime submodules $P_{1}$ and $%
P_{2}$ of $M$ and some positive integers $n,m\geq1$.
\end{enumerate}
\end{theorem}

\begin{proof}
By the fact that every multiplication module over a Noetherian ring is a
Noetherian module, $M$ is Noetherian and so finitely generated.\newline
$(1)\Rightarrow(2)$ Let $N$ be a 2-absorbing primary submodule of $M$. There
exists a proper ideal $I$ of $R$ such that $N=IM$. So $(N:_{R}M)=I$ is a $2$%
-absorbing primary ideal of $R$, by Theorem \ref{m-rad}. Since $R$ is a
Dedekind domain, then we have either $I=\mathfrak{m}^{n}$ for some maximal
ideal $\mathfrak{m}$ of $R$ and some positive integer $n\geq1$ or $I=%
\mathfrak{m}_{1}^{n}\mathfrak{m}_{2}^{m}$ for some maximal ideals $\mathfrak{%
m}_{1}$ and $\mathfrak{m}_{2}$ of $R$ and some positive integers $n,m\geq1$,
by \textrm{\cite[Theorem 2.11]{B}}. Thus, either $N=\mathfrak{m}^{n}M=(%
\mathfrak{m}M)^{n}$ or $N=(\mathfrak{m}_{1}M)^{n}(\mathfrak{m}_{2}M)^{m}$ as
desired. \newline
$(2)\Rightarrow(3)$ is clear.\newline
$(3)\Rightarrow(1)$ It is sufficient to show that $R$ is a Dedekind domain,
for this let $\mathfrak{m}$ be a maximal ideal of $R$. Let $I$ be an ideal
of $R$ such that $\mathfrak{m}^{2}\subset I\subset\mathfrak{m}$. So $\sqrt{I}%
=\mathfrak{m}$ and then $M$-$\mathrm{rad}(IM)=\mathfrak{m}M$, since $M$ is a
faithful multiplication $R$-module. Then $IM$ is a $2$-absorbing primary
submodule of $M$, Theorem \ref{rad}. By assumption, either $IM=P^{n}$ for
some prime submodule $P$ of $M$ and some positive integer $n\geq1$ or $%
IM=P_{1}^{n}P_{2}^{m}$ for some prime submodules $P_{1}$ and $P_{2}$ of $M$
and some positive integers $n,m\geq1$. Now, since $M$ is cancellation,
either $I=\mathfrak{p}^{n}$ for some prime ideal $\mathfrak{p}$ of $R$ or $I=%
\mathfrak{p}_{1}^{n}\mathfrak{p}_{2}^{m}$ for some prime ideals $\mathfrak{p}%
_{1}$ and $\mathfrak{p}_{2}$ of $R$, which any two cases have a
contradiction. Hence there are no ideals properly between $\mathfrak{m}^{2}$
and $\mathfrak{m}$. Consequently $R$ is a Dedekind domain by \textrm{\cite[%
Theorem 39.2, p. 470]{G}}.
\end{proof}

\begin{proposition}
\label{sqrt} Let $M$ be a multiplication $R$-module and $K,~N$ be submodules
of $M$. Then

\begin{enumerate}
\item $\sqrt{(KN:_{R}M)}=\sqrt{(K:_{R}M)}\cap\sqrt{(N:_{R}M)}$.

\item $M$-${rad}(KN)=M$-${rad}(K)\cap M$-${rad}(N)$.

\item $M$-${rad}(K\cap N)=M$-${rad}(K)\cap M$-${rad}(N)$.
\end{enumerate}
\end{proposition}

\begin{proof}
(1) By hypothesis there exist ideals $I,~J$ of $R$ such that $K=IM$ and $N=JM
$. Now assume $r\in\sqrt{(K:_{R}M)}\cap\sqrt{(N:_{R}M)}$. Therefore there
exist positive integers $m,~n$ such that $r^{m}M\subseteq IM$ and $%
r^{n}M\subseteq JM$. Hence $r^{m+n}M\subseteq r^{m}JM\subseteq IJM=KN$. So $%
r\in \sqrt{(KN:_{R}M)}$. Consequently $\sqrt{(K:_{R}M)}\cap\sqrt{(N:_{R}M)}%
\subseteq\sqrt{(KN:_{R}M)}$. The other inclusion trivially holds. \newline
$(2)$ By part (1) and \textrm{\cite[Corollary 1.7]{ES}}, 
\begin{eqnarray*}
M-\mathrm{rad}(KN) &=&\sqrt{(KN:_{R}M)}M=(\sqrt{(K:_{R}M)}\cap \sqrt{%
(N:_{R}M)})M \\
&=&([\sqrt{(K:_{R}M)}+\mathrm{ann}M]\cap \lbrack \sqrt{(N:_{R}M)}+\mathrm{ann%
}M])M \\
&=&\sqrt{(K:_{R}M)}M\cap \sqrt{(N:_{R}M)}M \\
&=&M-\mathrm{rad}(K)\cap M-\mathrm{rad}(N).
\end{eqnarray*}
$(3)$ See \textrm{\cite[Theorem 15(3)]{A}}.
\end{proof}

\begin{theorem}
Let $M$ be a multiplication $R$-module and $N_{1}, N_{2},...,N_{n}$ be
2-absorbing primary submodules of $M$ with the same $M$-radical. Then $%
N=\cap_{i=1}^{n}N_{i}$ is a 2-absorbing primary submodule of $M$.
\end{theorem}

\begin{proof}
Notice that $M$-$\mathrm{rad}(N)=\cap_{i=1}^{n}M$-$\mathrm{rad}(N_{i})$, by
Proposition \ref{sqrt}. Suppose that $abm\in N$ for some $a,b\in R$ and $%
m\in M$ and $ab\not\in(N:_{R}M)$. Then $ab\not\in(N_{i}:_{R}M)$ for some $%
1\leq i\leq n$. Hence $am\in M$-$\mathrm{rad}(N_{i})$ or $bm\in M$-$\mathrm{%
rad}(N_{i})$.
\end{proof}

\begin{lemma}
\label{element} Let $M$ be an $R$-module and $N$ a 2-absorbing primary
submodule of $M$. Suppose that $abK\subseteq N$ for some elements $a,b\in R$
and some submodule $K$ of $M$. If $ab\not\in(N:_{R}M)$, then $aK\subseteq M$-%
${rad}(N)$ or $bK\subseteq M$-${rad}(N)$.
\end{lemma}

\begin{proof}
Suppose that $aK\nsubseteq M$-$\mathrm{rad}(N)$ and $bK\nsubseteq M$-$%
\mathrm{rad}(N)$. Then $ak_{1}\not\in M$-$\mathrm{rad}(N)$ and $%
bk_{2}\not\in M$-$\mathrm{rad}(N)$ for some $k_{1},k_{2}\in K$. Since $%
abk_{1}\in N$ and $ab\not\in(N:_{R}M)$ and $ak_{1}\not\in M$-$\mathrm{rad}%
(N) $, we have $bk_{1}\in M$-$\mathrm{rad}(N)$. Since $abk_{2}\in N$ and $%
ab\not\in(N:_{R}M)$ and $bk_{2}\not\in M$-$\mathrm{rad}(N)$, we have $%
ak_{2}\in M$-$\mathrm{rad}(N)$. Now, since $ab(k_{1}+k_{2})\in N$ and $%
ab\not\in(N:_{R}M)$, we have $a(k_{1}+k_{2})\in M$-$\mathrm{rad}(N)$ or $%
b(k_{1}+k_{2})\in M$-$\mathrm{rad}(N)$. Suppose that $%
a(k_{1}+k_{2})=ak_{1}+ak_{2}\in M$-$\mathrm{rad}(N)$. Since $ak_{2}\in M$-$%
\mathrm{rad}(N)$, we have $ak_{1}\in M$-$\mathrm{rad}(N)$, a contradiction.
Suppose that $b(k_{1}+k_{2})=bk_{1}+bk_{2}\in M$-$\mathrm{rad}(N)$. Since $%
bk_{1}\in M$-$\mathrm{rad}(N)$, we have $bk_{2}\in M$-$\mathrm{rad}(N)$, a
contradiction again. Thus $aK\subseteq M$-$\mathrm{rad}(N)$ or $bK\subseteq
M $-$\mathrm{rad}(N)$.
\end{proof}

The following theorem offers a characterization of 2-absorbing primary
submodules.

\begin{theorem}
\label{main} Let $M$ be an $R$-module and $N$ be a proper submodule of $M$.
The following conditions are equivalent:

\begin{enumerate}
\item $N$ is a 2-absorbing primary submodule of $M$;

\item If $I_{1}I_{2}K\subseteq N$ for some ideals $I_{1},~I_{2}$ of $R$ and
some submodule $K$ of $M$, then either $I_{1}I_{2}\subseteq(N:_{R}M)$ or $%
I_{1}K\subseteq M$-${rad}(N)$ or $I_{2}K\subseteq M$-${rad}(N)$;

\item If $N_1N_2N_3\subseteq N$ for some submodules $N_1,~N_2$ and $N_3$ of $%
M$, then either $N_1N_2\subseteq N$ or $N_1N_3\subseteq M$-${rad}(N)$ or $%
N_2N_3\subseteq M$-${rad}(N)$.
\end{enumerate}
\end{theorem}

\begin{proof}
(1)$\Rightarrow$(2) Suppose that $N$ is a $2$-absorbing primary submodule of 
$M$ and $I_{1}I_{2}K\subseteq N$ for some ideals $I_{1},I_{2}$ of $R$ and
some submodule $K$ of $M$ and $I_{1}I_{2}\nsubseteq(N:_{R}M)$. We show that $%
I_{1}K\subseteq M$-$\mathrm{rad}(N)$ or $I_{2}K\subseteq M$-$\mathrm{rad}(N)$%
. Suppose that $I_{1}K\nsubseteq M$-$\mathrm{rad}(N)$ and $I_{2}K\nsubseteq
M $-$\mathrm{rad}(N)$. Then there are $a_{1}\in I_{1}$ and $a_{2}\in I_{2}$
such that $a_{1}K\nsubseteq M$-$\mathrm{rad}(N)$ and $a_{2}K\nsubseteq M$-$%
\mathrm{rad}(N)$. Since $a_{1}a_{2}K\subseteq N$ and neither $%
a_{1}K\subseteq M$-$\mathrm{rad}(N)$ nor $a_{2}K\subseteq M$-$\mathrm{rad}%
(N) $, we have $a_{1}a_{2}\in(N:_{R}M)$ by Lemma \ref{element}.\newline
Since $I_{1}I_{2}\nsubseteq(N:_{R}M)$, we have $b_{1}b_{2}\not\in(N:_{R}M)$
for some $b_{1}\in I_{1}$ and $b_{2}\in I_{2}$. Since $b_{1}b_{2}K\subseteq
N $ and $b_{1}b_{2}\not\in(N:_{R}M)$, we have $b_{1}K\subseteq M$-$\mathrm{%
rad}(N)$ or $b_{2}K\subseteq M$-$\mathrm{rad}(N)$ by Lemma \ref{element}. We
consider three cases.

\textbf{Case 1}. Suppose that $b_{1}K\subseteq M$-$\mathrm{rad}(N)$ but $%
b_{2}K\nsubseteq M$-$\mathrm{rad}(N)$. Since $a_{1}b_{2}K\subseteq N$ and
neither $b_{2}K\subseteq M$-$\mathrm{rad}(N)$ nor $a_{1}K\subseteq M$-$%
\mathrm{rad}(N)$, we conclude that $a_{1}b_{2}\in(N:_{R}M)$ by Lemma \ref%
{element}. Since $b_{1}K\subseteq M$-$\mathrm{rad}(N)$ but $a_{1}K\nsubseteq
M$-$\mathrm{rad}(N)$, we conclude that $(a_{1}+b_{1})K\nsubseteq M$-$\mathrm{%
rad}(N)$. Since $(a_{1}+b_{1})b_{2}K\subseteq N$ and neither $%
b_{2}K\subseteq M$-$\mathrm{rad}(N)$ nor $(a_{1}+b_{1})K\subseteq M$-$%
\mathrm{rad}(N)$, we conclude that $(a_{1}+b_{1})b_{2}\in (N:_{R}M)$ by
Lemma \ref{element}. Since $(a_{1}+b_{1})b_{2}=a_{1}b_{2}+b_{1}b_{2}\in
(N:_{R}M)$ and $a_{1}b_{2}\in (N:_{R}M)$, we conclude that $%
b_{1}b_{2}\in(N:_{R}M)$, a contradiction.

\textbf{Case 2}. Suppose that $b_{2}K\subseteq M$-$\mathrm{rad}(N)$ but $%
b_{1}K\nsubseteq M$-$\mathrm{rad}(N)$. Similar to the previous case we reach
to a contradiction.

\textbf{Case 3}. Suppose that $b_{1}K\subseteq M$-$\mathrm{rad}(N)$ and $%
b_{2}K\subseteq M$-$\mathrm{rad}(N)$. Since $b_{2}K\subseteq M$-$\mathrm{rad}%
(N)$ and $a_{2}K\nsubseteq M$-$\mathrm{rad}(N)$, we conclude that $%
(a_{2}+b_{2})K\nsubseteq M$-$\mathrm{rad}(N)$. Since $a_{1}(a_{2}+b_{2})K%
\subseteq N$ and neither $a_{1}K\subseteq M$-$\mathrm{rad}(N)$ nor $%
(a_{2}+b_{2})K\subseteq M$-$\mathrm{rad}(N)$, we conclude that $%
a_{1}(a_{2}+b_{2})=a_{1}a_{2}+a_{1}b_{2}\in(N:_{R}M)$ by Lemma \ref{element}%
. Since $a_{1}a_{2}\in(N:_{R}M)$ and $a_{1}a_{2}+a_{1}b_{2}\in(N:_{R}M)$, we
conclude that $a_{1}b_{2}\in(N:_{R}M)$. Since $b_{1}K\subseteq M$-$\mathrm{%
rad}(N)$ and $a_{1}K\nsubseteq M$-$\mathrm{rad}(N)$, we conclude that $%
(a_{1}+b_{1})K\nsubseteq M$-$\mathrm{rad}(N)$. Since $(a_{1}+b_{1})a_{2}K%
\subseteq N$ and neither $a_{2}K\subseteq M$-$\mathrm{rad}(N)$ nor $%
(a_{1}+b_{1})K\subseteq M$-$\mathrm{rad}(N)$, we conclude that $%
(a_{1}+b_{1})a_{2}=a_{1}a_{2}+b_{1}a_{2}\in(N:_{R}M)$ by Lemma \ref{element}%
. Since $a_{1}a_{2}\in(N:_{R}M)$ and $a_{1}a_{2}+b_{1}a_{2}\in(N:_{R}M)$, we
conclude that $b_{1}a_{2}\in(N:_{R}M)$. Now, since $%
(a_{1}+b_{1})(a_{2}+b_{2})K\subseteq N$ and neither $(a_{1}+b_{1})K\subseteq
M$-$\mathrm{rad}(N)$ nor $(a_{2}+b_{2})K\subseteq M$-$\mathrm{rad}(N)$, we
conclude that $%
(a_{1}+b_{1})(a_{2}+b_{2})=a_{1}a_{2}+a_{1}b_{2}+b_{1}a_{2}+b_{1}b_{2}%
\in(N:_{R}M)$ by Lemma \ref{element}. Since $%
a_{1}a_{2},a_{1}b_{2},b_{1}a_{2}\in(N:_{R}M)$, we have $b_{1}b_{2}%
\in(N:_{R}M)$, a contradiction. Consequently $I_{1}K\subseteq M$-$\mathrm{rad%
}(N)$ or $I_{2}K\subseteq M$-$\mathrm{rad}(N)$.\newline
(2)$\Rightarrow$(1) is trivial.\newline
(2)$\Rightarrow$(3) Let $N_1N_2N_3\subseteq N$ for some submodules $N_1,~N_2$
and $N_3$ of $M$ such that $N_1N_2\nsubseteq N$. Since $M$ is
multiplication, there are ideals $I_1,~I_2$ of $R$ such that $N_1=I_1M$, $%
N_2=I_2M$. Clearly $I_1I_2N_3\subseteq N$ and $I_1I_2\nsubseteq(N:_RM)$.
Therefore $I_1N_3\subseteq M$-$\mathrm{rad}(N)$ or $I_2N_3\subseteq M$-$%
\mathrm{rad}(N)$, which implies that $N_1N_3\subseteq M$-$\mathrm{rad}(N)$
or $N_2N_3\subseteq M$-$\mathrm{rad}(N)$.\newline
(3)$\Rightarrow$(2) Suppose that $I_{1}I_{2}K\subseteq N$ for some ideals $%
I_{1},~I_{2}$ of $R$ and some submodule $K$ of $M$. It is sufficient to set $%
N_1:=I_1M$, $N_2:=I_2M$ and $N_3=K$ in part (3).
\end{proof}

\begin{theorem}
\label{conve} Let $M$ be a multiplication $R$-module and $N$ a submodule of $%
M$. If $(N:_{R}M)$ is a 2-absorbing primary ideal of $R$, then $N$ is a
2-absorbing primary submodule of $M$.
\end{theorem}

\begin{proof}
Let $I_1I_2K\subseteq N$ for some ideals $I_1,~I_2$ of $R$ and some
submodule $K$ of $M$. Since $M$ is multiplication, then there is an ideal $%
I_3$ of $R$ such that $K=I_3M$. Hence $I_1I_2I_3\subseteq (N:_{R}M)$ which
implies that either $I_1I_2\subseteq (N:_{R}M)$ or $I_1I_3\subseteq \sqrt{%
(N:_{R}M)}$ or $I_2I_3\subseteq \sqrt{(N:_{R}M)}$, by \textrm{\cite[Theorem
2.19]{Bt}}. If $I_1I_2\subseteq (N:_{R}M),$ then we are done. So, suppose
that $I_1I_3\subseteq\sqrt{(N:_{R}M)}$. Thus $I_1I_3M=I_1K\subseteq \sqrt{%
(N:_{R}M)}M=M$-$\mathrm{rad}(N)$. Similary if $I_2I_3\subseteq \sqrt{%
(N:_{R}M)}$, then we have $I_2K\subseteq M$-$\mathrm{rad}(N)$. It completes
the proof, by Theorem \ref{main}.
\end{proof}

The following example shows that Theorem \ref{conve} is not satisfied in
general.

\begin{example}
Consider the $\mathbb{Z} $-module $M=\mathbb{Z}\times \mathbb{Z}$ and $N=6%
\mathbb{Z} \times 0$ a submodule of $M$. Observe that $\mathbb{Z} \times 0$, 
$2\mathbb{Z} \times \mathbb{Z}$ and $3\mathbb{Z} \times \mathbb{Z}$ are some
of the prime submodules of $M$ containing $N$. Also $(N:_{\mathbb{Z}}M)=0$
is a 2-absorbing primary ideal of $\mathbb{Z} $. On the other hand, since $%
2.3.(1,0)=(6,0)\in N$, $2.3\notin (N:_{\mathbb{Z}}M)$, $2.(1,0)=(2,0)\notin
M $-$\mathrm{rad}(N)\subseteq (\mathbb{Z} \times 0)\cap (2\mathbb{Z} \times 
\mathbb{Z})\cap (3\mathbb{Z} \times \mathbb{Z})=6\mathbb{Z} \times 0=N$ and $%
3.(1,0)=(3,0)\notin M$-$\mathrm{rad}(N)=N$, so $N$ is not a 2-absorbing
primary submodule of $M$.
\end{example}

\begin{theorem}
\label{prod} Let $M$ be a multiplication $R$-module and $N_{1}$ and $N_{2}$
be primary submodules of $M$. Then $N_{1}\cap N_{2}$ is a 2-absorbing
primary submodule of $M$. If in addition $M$ is finitely generated faithful,
then $N_{1}N_{2}$ is a 2-absorbing primary submodule of $M$.
\end{theorem}

\begin{proof}
Since $N_{1}$ and $N_{2}$ are primary submodules of $M$, then $(N_{1}:_{R}M)$
and $(N_{2}:_{R}M)$ are primary ideals of $R$. Hence $%
(N_{1}:_{R}M)(N_{2}:_{R}M)$ and $(N_{1}\cap
N_{2}:_{R}M)=(N_{1}:_{R}M)\cap(N_{2}:_{R}M)$ are 2-absorbing primary ideals
of $R$, by \textrm{\cite[Theorem 2.4]{Bt}}. Therefore, Theorem \ref{conve}
implies that $N_{1}\cap N_{2}$ is a 2-absorbing primary submodule of $M$. If 
$M$ is a finitely generated faithful multiplication $R$-module, then $%
(N_{1}N_{2}:_{R}M)=(N_{1}:_{R}M)(N_{2}:_{R}M)$. So, again by Theorem \ref%
{conve} we deduce that $N_{1}N_{2}$ is a 2-absorbing primary submodule of $M$%
.
\end{proof}

Let $M$ be a multiplication $R$-module and $N$ a primary submodule of $M$.
We know that $\sqrt{(N:_{R}M)}$ is a prime ideal of $R$ and so $P=M$-$%
\mathrm{rad}(N)=\sqrt{(N:_{R}M)}M$ is a prime submodule of $M$. In this case
we say that $N$ is a $P$-\textit{primary submodule of} $M$.

\begin{corollary}
\label{cprod} Let $M$ be a multiplication $R$-module and $P_{1}$ and $P_{2}$
be prime submodules of $M$. Suppose that $P_{1}^{n}$ is a $P_{1}$-primary
submodule of $M$ for some positive integer $n\geq1$ and $P_{2}^{m}$ is a $%
P_{2}$-primary submodule of $M$ for some positive integer $m\geq1$.

\begin{enumerate}
\item $P_{1}^{n}\cap P_{2}^{m}$ is a 2-absorbing primary submodule of $M$.

\item If in addition $M$ is finitely generated faithful, then $%
P_{1}^{n}P_{2}^{m}$ is a 2-absorbing primary submodule of $M$.
\end{enumerate}
\end{corollary}

\begin{theorem}
Let $M$ be a multiplication $R$-module and $N$ be a submodule of $M$ that has a primary
decomposition. If $M$-$rad(N)=\mathfrak{M_1}\cap\mathfrak{M_2}$ where $%
\mathfrak{M_1}$ and $\mathfrak{M_2}$ are two maximal submodules of $M$, then 
$N$ is a 2-absorbing primary submodule of $M$.
\end{theorem}

\begin{proof}
Assume that $N=N_1\cap\cdots\cap N_n$ is a primary decomposition. By
Proposition \ref{sqrt}(3), $M$-$\mathrm{rad}(N)=M$-$\mathrm{rad}%
(N_1)\cap\cdots\cap M$-$\mathrm{rad}(N_n)=\mathfrak{M_1}\cap\mathfrak{M_2}$.
Since $M$-$\mathrm{rad}(N_i)$'s are prime submodules of $M$, then $\{M$-$%
\mathrm{rad}(N_1),...,M$-$\mathrm{rad}(N_n)\}= \{\mathfrak{M_1},\mathfrak{M_2%
}\}$, by \cite[Theorem 3.16]{Am}. Without loss of generality we may assume
that for some $1\leq t< n$, $\{M$-$\mathrm{rad}(N_1),...,M$-$\mathrm{rad}%
(N_t)\}= \{\mathfrak{M_1}\}$ and $\{M$-$\mathrm{rad}(N_{t+1}),...,M$-$%
\mathrm{rad}(N_n)\}= \{\mathfrak{M_2}\}$. Set $K_1:=N_1\cap\cdots\cap N_t$
and $K_2:=N_{t+1}\cap\cdots\cap N_n$. By \cite[Lemma 1.2.2]{Ash}, $K_1$ is
an $\mathfrak{M_1}$-primary submodule and $K_2$ is an $\mathfrak{M_2}$%
-primary submodule of $M$. Therefore, by Theorem \ref{prod}, $N=K_1\cap K_2$ is
2-absorbing primary.
\end{proof}

\begin{lemma}
$($\textrm{\cite[Corollary 1.3]{Mcc}}$)$\label{equal} Let $M$ and $%
M^{\prime} $ be $R$-modules with $f:M\to M^{\prime}$ an $R$-module
epimorphism. If $N$ is a submodule of $M$ containing $Ker(f)$, then $f(M$-${%
rad}(N))=M^{\prime}$-${rad}(f(N))$.
\end{lemma}

\begin{theorem}
\label{im} Let $f:M\to M^{\prime}$ be a homomorphism of $R$-modules.

\begin{enumerate}
\item If $N^{\prime}$ is a 2-absorbing primary submodule of $M^{\prime}$,
then $f^{-1}(N^{\prime})$ is a 2-absorbing primary submodule of $M$.

\item If $f$ is epimorphism and $N$ is a 2-absorbing primary submodule of $M$
containing $Ker(f)$, then $f(N)$ is a 2-absorbing primary submodule of $%
M^{\prime}$.
\end{enumerate}
\end{theorem}

\begin{proof}
$(1)$ Let $a,b\in R$ and $m\in M$ such that $abm\in f^{-1}(N^{\prime})$.
Then $abf(m)\in N^{\prime}$. Hence $ab\in(N^{\prime}:_{R}M^{\prime})$ or $%
af(m)\in M^{\prime}$-$\mathrm{rad}(N^{\prime})$ or $bf(m)\in M^{\prime}$-$%
\mathrm{rad}(N^{\prime})$, and thus $ab\in(f^{-1}(N^{\prime}):_{R}M)$ or $%
am\in f^{-1}(M^{\prime}$-$\mathrm{rad}(N^{\prime}))$ or $bm\in
f^{-1}(M^{\prime}$-$\mathrm{rad}(N^{\prime}))$. By using the inclusion $%
f^{-1}(M^{\prime}$-$\mathrm{rad}(N^{\prime}))\subseteq M$-$\mathrm{rad}%
(f^{-1}(N^{\prime}))$, we conclude that $f^{-1}(N^{\prime})$ is a $2$%
-absorbing primary submodule of $M$.\newline
$(2)$ Let $a,b\in R$, $m^{\prime}\in M^{\prime}$ and $abm^{\prime}\in f(N)$.
By assumption there exists $m\in M$ such that $m^{\prime}=f(m)$ and so $%
f(abm)\in f(N)$. Since $Ker(f)\subseteq N$, we have $abm\in N$. It implies
that $ab\in(N:_{R}M)$ or $am\in M$-$\mathrm{rad}(N)$ or $bm\in M$-$\mathrm{%
rad}(N)$. Hence $ab\in(f(N):_{R}M^{\prime})$ or $am^{\prime}\in f(M$-$%
\mathrm{rad}(N))=M^{\prime}$-$\mathrm{rad}(f(N))$ or $bm^{\prime}\in f(M$-$%
\mathrm{rad}(N))=M^{\prime}$-$\mathrm{rad}(f(N))$. Consequently $f(N)$ is a $%
2$-absorbing primary submodule of $M^{\prime}$.
\end{proof}

As an immediate consequence of Theorem \ref{im}(2) we have the following
Corollary.

\begin{corollary}
\label{quo} Let $M$ be an $R$-module and $L\subseteq N$ be submodules of $M$%
. If $N$ is a $2$-absorbing primary submodule of $M$, then $N/L$ is a
2-absorbing primary submodule of $M/L$.
\end{corollary}

\begin{theorem}
Let $K$ and $N$ be submodules of $M$ with $K\subset N\subset M$. If $K$ is a
2-absorbing primary submodule of $M$ and $N/K$ is a weakly 2-absorbing
primary submodule of $M/K$, then $N$ is a 2-absorbing primary submodule of $M
$.
\end{theorem}

\begin{proof}
Let $a,b\in R$, $m\in M$ and $abm\in N$. If $abm\in K$, then $am\in M$-$%
\mathrm{rad}(K)\subseteq M$-$\mathrm{rad}(N)$ or $bm\in M$-$\mathrm{rad}%
(K)\subseteq M$-$\mathrm{rad}(N)$ or $ab\in (K:_{R}M)\subseteq (N:_{R}M)$ as
it is needed.\newline
So suppose that $abm\not\in K$. Then $0\neq ab(m+K)\in N/K$ that implies, $%
a(m+K)\in M/K$-$\mathrm{rad}(N/K)=\frac{M-\mathrm{rad}(N)}{K}$ or $b(m+K)\in
M/K$-$\mathrm{rad}(N/K)$ or $ab\in (N/K:_{R}M/K)$. It means that $am\in M$-$%
\mathrm{rad}(N)$ or $bm\in M$-$\mathrm{rad}(N)$ or $ab\in (N:_{R}M)$, which
completes the proof.
\end{proof}

Let $R_i$ be a commutative ring with identity and $M_i$ be an $R_i$-module,
for $i = 1, 2$. Let $R=R_{1}\times R_{2}$. Then $M=M_{1}\times M_{2}$ is an $%
R$-module and each submodule of $M$ is of the form $N=N_{1}\times N_{2}$ for
some submodules $N_1$ of $M_1$ and $N_2$ of $M_2$. In addition, if $M_i$ is
a multiplication $R_i$-module, for $i = 1, 2$, then $M$ is a multiplication $%
R$-module. In this case, for each submodule $N=N_{1}\times N_{2}$ of $M$ we
have $M$-$\mathrm{rad}(N)=M_{1}$-$\mathrm{rad}(N_{1})\times M_{2}$-$\mathrm{%
rad}(N_{2})$.

\begin{theorem}
\label{cart}Let $R=R_{1}\times R_{2}$ and $M=M_{1}\times M_{2}$ where $M_{1}$
is a multiplication $R_{1}$-module and $M_{2}$ is a multiplication $R_{2}$%
-module.

\begin{enumerate}
\item A proper submodule $K_{1}$ of $M_{1}$ is a 2-absorbing primary
submodule if and only if $N=K_{1}\times M_{2}$ is a 2-absorbing primary
submodule of $M$.

\item A proper submodule $K_{2}$ of $M_{2}$ is a 2-absorbing primary
submodule if and only if $N=M_{1}\times K_{2}$ is a 2-absorbing primary
submodule of $M$.

\item If $K_{1}$ is a primary submodule of $M_{1}$ and $K_{2}$ is a primary
submodule of $M_{2}$, then $N=K_{1}\times K_{2}$ is a 2-absorbing primary
submodule of $M.$
\end{enumerate}
\end{theorem}

\begin{proof}
(1) Suppose that $N=K_{1}\times M_{2}$\ is a 2-absorbing primary submodule
of $M$. From our hypothesis, $N$ is proper, so $K_{1}\not=M_{1}.$ Set $%
M^{\prime }=\frac{M}{\{0\}\times M_{2}}$. Hence $N^{\prime }=\frac{N}{%
\{0\}\times M_{2}}$ is a 2-absorbing primary submodule of $M^{\prime }$ by
Corollary \ref{quo}. Also observe that $M^{\prime }\cong M_{1}$ and $%
N^{\prime }\cong K_{1}$. Thus $K_{1}$ is a 2-absorbing primary submodule of $%
M_{1}.$ Conversely, if $K_{1}$ is a 2-absorbing primary submodule of $M_{1},$
then it is clear that $N=K_{1}\times M_{2}$ is a 2-absorbing primary
submodule of $M$.\newline
(2) It can be easily verified similar to (1).\newline
(3) Assume that $N=K_{1}\times K_{2}$ where $K_{1}$ and $K_{2}$ are primary
submodules of $M_{1}$ and $M_{2}$, respectively. Hence $(K_{1}\times
M_{2})\cap (M_{1}\times K_{2})=K_{1}\times K_{2}=N$ is a 2-absorbing primary
submodule of $M$, by parts (1) and (2) and Theorem \ref{prod}.
\end{proof}

\begin{theorem}
Let $R=R_{1}\times R_{2}$ and $M=M_{1}\times M_{2}$ be a finitely generated
multiplication $R$-module where $M_{1}$ is a multiplication $R_{1}$-module
and $M_{2}$ is a multiplication $R_{2}$-module. If $N=N_{1}\times N_{2}$ is
a proper submodule of $M$, then the followings are equivalent.

\begin{enumerate}
\item $N$ is a 2-absorbing primary submodule of $M.$

\item $N_{1}=M_{1}$ and $N_{2}$ is a 2-absorbing primary submodule of $M_{2}$
or $N_{2}=M_{2}$ and $N_{1}$ is a 2-absorbing primary submodule of $M_{1}$
or $N_{1},$ $N_{2}$ are primary submodules of $M_{1}$, $M_{2},$ respectively.
\end{enumerate}
\end{theorem}

\begin{proof}
(1)$\Rightarrow $(2) Suppose that $N=N_{1}\times N_{2}$ is a 2-absorbing
primary submodule of $M$. Then $(N:M)=(N_{1}:M_{1})\times (N_{2}:M_{2})$ is
a 2-absorbing primary ideal of $R=R_{1}\times R_{2}$ by Theorem \ref{m-rad}.
From Theorem 2.3 in \cite{Bt}, we have $(N_{1}:M_{1})=R_{1}$ and $%
(N_{2}:M_{2})$ is a 2-absorbing primary ideal of $R_{2}~$or $%
(N_{2}:M_{2})=R_{2}$ and $(N_{1}:M_{1})$ is a 2-absorbing primary ideal of $%
R_{1}$ or $(N_{1}:M_{1})$ and $(N_{2}:M_{2})$ are primary ideals of $R_{1},$ 
$R_{2},$ respectively. Assume that $(N_{1}:M_{1})=R_{1}$ and $(N_{2}:M_{2})$
is a 2-absorbing primary ideal of $R_{2}.$ Thus $N_{1}=M_{1}$ and $N_{2}$ is
a 2-absorbing primary submodule of $M_{2}$ by Theorem \ref{conve}. Similarly
if $(N_{2}:M_{2})=R_{2}$ and $(N_{1}:M_{1})$ is a 2-absorbing primary ideal
of $R_{1},$ then $N_{2}=M_{2}$ and $N_{1}$ is a 2-absorbing primary
submodule of $M.$ And if the last case hold, then clearly we conclude that $%
N_{1},$ $N_{2}$ are primary submodules of $M_{1}$, $M_{2},$ respectively.\\
(2)$\Rightarrow $(1) \ It is clear from Theorem \ref{cart}.
\end{proof}


\begin{thebibliography}{99}
\bibitem{A} M. M. Ali, Idempotent and nilpotent submodules of multiplication
modules, Comm. Algebra, \textbf{36} (2008), 4620--4642.

\bibitem{Al} M. M. Ali, Invertiblity of Multiplication Modules III, New
Zeland J. Math., \textbf{39} (2009), 193-213.

\bibitem{Am} R. Ameri, On the prime submodules of multiplication modules,
Inter. J. Math. Math. Sci., \textbf{27} (2003), 1715--1724.

\bibitem{AB} D. D. Anderson and M. Batanieh, Generalizations of prime
ideals, Comm. Algebra, \textbf{36} (2008), 686--696.

\bibitem{AS} D. D. Anderson and E. Smith, Weakly prime ideals, Houston J.
Math., \textbf{29} (2003), 831--840.

\bibitem{AB1} D. F. Anderson and A. Badawi, On $n$-absorbing ideals of
commutative rings, Comm. Algebra, \textbf{39} (2011), 1646--1672.

\bibitem{AF} F. Anderson and K. Fuller, Rings and categories of modules.
New-York: Springer-Verlag, 1992.

\bibitem{Ash} R. B. Ash, A course in commutative algebra. University of
Illinois, 2006.

\bibitem{B} A. Badawi, On $2$-absorbing ideals of commutative rings, Bull.
Austral. Math. Soc., \textbf{75} (2007), 417--429.

\bibitem{YB} A. Badawi and A. Yousefian Darani, On weakly $2$-absorbing
ideals of commutative rings, Houston J. Math., \textbf{39} (2013), 441--452.

\bibitem{Bt} A. Badawi, \"{U}. Tekir and E. Yetkin, On $2$-absorbing primary
ideals in commutative rings, Bull. Korean Math. Soc., \textbf{51} (4)
(2014), 1163--1173.

\bibitem{Ba} A. Barnard, Multiplication modules, J. Algebra, \textbf{71}
(1981), 174--178.

\bibitem{BK} M. Behboodi and H. Koohi, Weakly prime modules, Vietnam J.
Math., \textbf{32} (2) (2004), 185--195.

\bibitem{BS} S. M. Bhatwadekar and P. K. Sharma, Unique factorization and
birth of almost primes, Comm. Algebra, \textbf{33} (2005), 43--49.

\bibitem{EC} S. Ebrahimi Atani, F. \c{C}allialp and \"{U}. Tekir, A Short
note on the primary submodules of multiplication modules, Inter. J. Algebra, 
\textbf{8} (1) (2007), 381--384.

\bibitem{E1} S. Ebrahimi Atani and F. Farzalipour, On weakly prime
submodules, Tamk. J. Math., \textbf{38} (3) (2007), 247--252.

\bibitem{ES} Z. A. El-Bast and P. F. Smith, Multiplication modules, Comm.
Algebra, \textbf{16} (1988), 755--779.

\bibitem{F} C. Faith, Algebra: Rings, modules and categories I,
Springer-Verlag Berlin Heidelberg New York 1973.

\bibitem{G} R. Gilmer, Multiplicative ideal theory. Queens Papers Pure Appl.
Math. 90, Queens University, Kingston (1992).

\bibitem{Kh} M. Khoramdel and S. Dolati Pish Hesari, Some notes on Dedekind
modules, Hacettepe J. Math. Stat., \textbf{40} (2011), 627--634.

\bibitem{MM} R. L. McCasland and M. E. Moore, On radicals of submodules of
finitely generated modules, Canad. Math. Bull., \textbf{29} (1986), 37--39.

\bibitem{Mcc} R. L. McCasland and M. E. Moore, Radicals of submodules, Comm.
Algebra, \textbf{19} (1991), 1327-1341.

\bibitem{NA} A. G. Naoum and F. H. Al-Alwan, Dedekind modules, Comm.
Algebra, \textbf{24} (2)(1996), 397-412.

\bibitem{p} Sh. Payrovi and S. Babaei, On 2-absorbing submodules, Algebra
Colloq., \textbf{19} (2012), 913-920.

\bibitem{S} {P. F. Smith}, Some remarks on multiplication modules, Arch.
Math., \textbf{50} (1988), 223--235.

\bibitem{YF1} A. Yousefian Darani and F. Soheilnia, On $2$-absorbing and
weakly $2$-absorbing submodules, Thai J. Math., \textbf{9} (2011), 577-584

\bibitem{YF2} A. Yousefian Darani and F. Soheilnia, On $n$-absorbing
submodules, Math. Commun., \textbf{17} (2012), 547-557.
\end{thebibliography}
\end{document}